\documentclass[11pt,twoside]{amsart}

\usepackage{amsmath}
\usepackage{amsthm}
\usepackage{amsfonts, amssymb}
\usepackage{mathrsfs}
\usepackage[all]{xy}
\usepackage{url}

\usepackage{latexsym}

\usepackage[pdftex]{graphicx}
\usepackage{subfig}
\usepackage{caption}
\usepackage{float}
\theoremstyle{plain}
\newtheorem{lema}{Lemma}[section]
\newtheorem{prop}[lema]{Proposition}

\newtheorem{teo}[lema]{Theorem}
\newtheorem{conj}[lema]{Conjecture}

\newtheorem{coro}[lema]{Corollary}

\theoremstyle{remark}

\newtheorem{obs}[lema]{Remark}

\theoremstyle{definition}
\newtheorem{defi}[lema]{Definition}
\newtheorem{ej}[lema]{Example}

\def\k{\mathcal{K}}

\def\etft0{\mathnormal{etfT_0}}

\def\dim{\text{dim}}
\def\he{\text{h}}

\newcommand{\st}{\operatorname{st}}
\newcommand{\lk}{\operatorname{lk}}

\newcommand{\K}{\mathcal{K}}
\newcommand{\X}{\mathcal{X}}

\begin{document}

\title[A new approach to Whitehead's asphericity question]{A new approach to Whitehead's asphericity question}

\author[M.A. Cerdeiro]{Manuela Ana Cerdeiro}
\author[E.G. Minian]{Elias Gabriel Minian}
\thanks{The authors' research is partially supported by CONICET}

\address{Departamento  de Matem\'atica - IMAS\\
 FCEyN, Universidad de Buenos Aires. \\ Buenos
Aires, Argentina.}

\email{cerdeiro@dm.uba.ar} 
\email{gminian@dm.uba.ar}

\begin{abstract}
We investigate Whitehead's asphericity question from a new perspective, using results and techniques of the homotopy theory of finite topological spaces. We also introduce a method of reduction to investigate asphericity based on the interaction between the combinatorics and the topology of finite spaces. 
\end{abstract}



\maketitle

\section{Introduction}

In 1941 J. H. C. Whitehead \cite{wh} stated the following question:
\begin{center}
\textit{Is every subcomplex of an aspherical 2-dimensional complex itself aspherical?}
\end{center}
Recall that a path connected space $X$ is called \textit{aspherical} if its homotopy groups $\pi_n(X)$ are trivial for all $n\geq 2$.
By the Hurewicz theorem applied to the universal cover, a connected $2$-dimensional CW-complex $X$ is aspherical if $\pi_2(X)$ is trivial. Whitehead's asphericity  question is still unanswered, although some interesting advances have been obtained along the last seventy years. In 1954 W. H. Cockroft proved that the answer to the question is \textit{positive} if the subcomplex $K\subset L$ is finite, $L$ is obtained from $K$ by adding $2$-cells and the group $\pi_1(K)$ is either Abelian, finite or free \cite{coc}. In 1983 J. Howie \cite{ho} reduced the problem to the following two particular cases.
\begin{teo} [Howie]\label{ho}
 If the answer to Whitehead's question is {\normalfont negative}, then there exists a counterexample $K\subset L$ of one of the following two types:
\begin{itemize}
 \item[(a)] $L$ is finite and contractible, $K=L-e$ for some 2-cell $e$ of $L$, and $K$ is non-aspherical.
 \item[(b)] $L$ is the union of an infinite chain of finite non-aspherical subcomplexes $K=K_0\subset K_1\subset \cdots$ such that each inclusion map $K_i\to K_{i+1}$ is nullhomotopic.
\end{itemize}
\end{teo}
In 1996 E. Luft \cite{lu} proved that the existence of a counterexample of type $(a)$ actually implies the existence of a counterexample of type $(b)$.

In this paper we investigate Whitehead's asphericity question from a new perspective, using some of the tools of the homotopy theory of finite topological spaces developed in \cite{ba,bami1}. We restate the problem in terms of finite topological spaces (see Conjecture \ref{conj_etf}) and prove that certain classes of finite spaces (and consequently, of compact complexes) satisfy the conjecture. There is a connection between Whitehead's asphericity question and the Andrews-Curtis conjecture (see \cite{ho}). In fact, Howie proved that if the Andrews-Curtis conjecture and the ribbon disc complements conjecture are true, then there are no counterexamples of type $(a)$. Recently Barmak introduced the notion of a \it quasi-constructible  \rm 2-complex and proved that the class of these 2-complexes, which contains all constructible 2-complexes, satisfies the Andrews-Curtis conjecture \cite{ba}. In this article we show that this extensive class contains no counterexamples of type (a). 

 We also introduce a method of reduction to study asphericity using the interaction between the combinatorics and the topology of finite spaces.

For a comprehensive exposition on Whitehead's asphericity question we refer the reader to \cite{bo}.

\section{Preliminaries}\label{preliminaries}

In this section we will review some basic notions and results on finite topological spaces. For further details on the theory of finite spaces and its applications we refer the reader to \cite{ba,bami1,ma,mc,st}.

There is a natural connection between finite topological spaces and finite posets. Given a finite set $X$ and a topology on $X$, one can define a relation on the set $X$ by $$x\leq y \Leftrightarrow x\in U\ \text{for every open set } U\  
\text{containing } y.$$ It is easy to see that this relation is a preorder, i.e. it is reflexive and transitive. Note that it is antisymmetric if and only if the topology satisfies the $T_0$ axiom: for any given pair of points of $X$, there is an open set which contains only one of them.

Conversely, given a preorder $\leq$ and an element $x\in X$, we set  $$U_x:=\{y\in X :\ y\leq x\}\subseteq X.$$ The sets $U_x$ (for $x\in X$) form a basis for a topology on $X$.
These applications are mutually inverse and give rise to a one-to-one correspondence between finite $T_0$ spaces and finite posets. From now on we will only consider finite spaces which satisfy the $T_0$ axiom. In order to represent a finite $T_0$-space, we will use the Hasse diagram of the corresponding poset.

It is easy to see that a map $f:X\to Y$ is continuous if and only if it is order preserving, so this correspondence extends to the morphisms. It is also clear that a finite space is connected if and only if it is path connected (see \cite[Prop. 1.2.4]{ba}).

The \textit{length} of a chain $x_0<x_1<\cdots<x_n$ in a finite space $X$ is  $n$. The \textit{height} $\he(X)$ is the maximum length of the chains of $X$.

Given a finite space $X$, the \textit{order complex} $\K(X)$ is the simplicial complex defined as follows. The vertices of $\K(X)$ are the elements of $X$ and the $n$-simplices of $\K(X)$ are the chains in $X$ of $n+1$ elements. There is also an application in the opposite direction, which assigns  to any simplicial complex $K$ the \textit{face poset} $\X(K)$ of simplices of $K$ ordered by inclusion.
It is clear that $\he(X)=\dim(\K(X))$ for every finite space $X$ and $\he(\X(K))=\dim(K)$ for every simplicial complex $K$.

Figure \ref{fig:figurecomplejo} shows a finite space $X$ of 6 points, represented by its Hasse diagram, and the associated 2-complex $\K(X)$.
 
\begin{figure}[h]
\begin{minipage}{0.5\linewidth}
  \begin{displaymath}
  \xymatrix@C=10pt{
   & \bullet \ar@{-}[dl] \ar@{-}[dr]&  \bullet \ar@{-}[d] \\
    \bullet \ar@{-}[d] \ar@{-}[drr]& &  \bullet \ar@{-}[dll] \ar@{-}[d] \\
    \bullet & &  \bullet}
  \end{displaymath}  
      \hspace{100pt} $X$
\end{minipage}\noindent
\begin{minipage}{0.5\linewidth}
\vspace{9pt}

 \includegraphics[angle=270,scale=0.3]{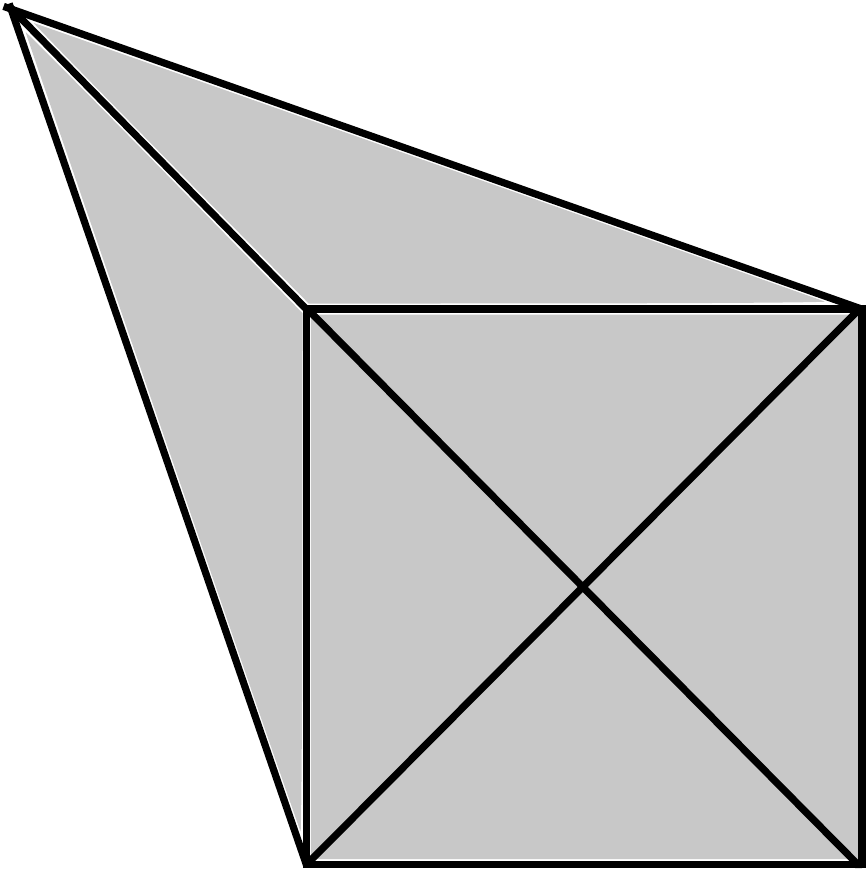}\\

\vspace{0pt}\hspace{7pt} $\K(X)$
\end{minipage}

\caption{The order complex of a finite space.}
\label{fig:figurecomplejo}
\end{figure}

A continuous function $f:X\to Y$ gives rise to a simplicial map $\K(f):\K(X)\to \K(Y)$ in the obvious way, and a simplicial map between complexes $\varphi:K\to L$ induces a continuous function $\X(\varphi):\X(K)\to \X(L)$.

The functors $\K$ and $\X$ are not mutually inverse. In fact, if $K$ is a simplicial complex, then $\K(\X(K))$ is the barycentric subdivision $K'$ of $K$. McCord proved in \cite{mc} that there are weak homotopy equivalences $\K(X)\to X$ and $K\to \X(K)$. Given two finite spaces $X$ and $Y$, they are weak homotopy equivalent if and only if their associated complexes are homotopy equivalent. Moreover, two finite simplicial complexes are homotopy equivalent if and only if their face posets are weak equivalent (viewed as finite spaces). Therefore the weak homotopy types of finite spaces cover all the homotopy types of compact polyhedra. At this point, it is important to notice that the Whitehead theorem does not hold for finite spaces. Counterexamples can be found in \cite{ba,bami1}. Concretely, there are finite spaces which are homotopically trivial (i.e. weak homotopy equivalent to the singleton) but not contractible.

We recall now  some of the main results of the homotopy theory of finite spaces.

A point $x\in X$ is called a \textit{down beat point} if $\hat{U}_x=\{y\in X:y<x\}$ has a maximum and an \textit{up beat point} if $\hat{F}_x=\{y\in X:y>x\}$ has a minimum. In both cases the subspace $X-x\subset X$ is a strong deformation retract. Furthermore, two finite spaces $X$ and $Y$ have the same homotopy type if and only if $Y$ can be obtained from $X$ by removing and adding beat points. The characterization of homotopy equivalent spaces in terms of beat points is due to Stong \cite{st}.

A point $x\in X$ is called a \textit{weak point} if $\hat{U}_x$ or $\hat{F}_x$ is a contractible finite space, or equivalently, if $\hat C_x=\{y\in X: y< x \text{ or } y> x\}$ is contractible. In this case the inclusion $X-x\hookrightarrow X$ is a weak homotopy equivalence. This notion was introduced in \cite{bami1} to study the simple homotopy theory of finite spaces and its applications to problems of simple homotopy theory of polyhedra (see\cite{ba,bami1}).

Since finite spaces with maximum or minimum are contractible, the notion of weak point generalizes that of beat point.
The process of removing a weak point is called an \textit{elementary collapse}, and a space $X$ is called \textit{collapsible} if there is a sequence of elementary collapses that transforms $X$ into a single point.
Contractible finite spaces are collapsible and collapsible spaces are homotopically trivial. None of the converse implications hold (see \cite{ba,bami1} for more details).

\subsection{QC-reducible spaces and quasi-constructible complexes}

The notions of qc-reduction and qc-reducible space were introduced by Barmak to investigate the Andrews-Curtis conjecture using finite spaces. 
We recall here these notions and we refer the reader to \cite{ba} for a detailed exposition on qc-reducible spaces and the Andrews-Curtis conjecture. 

\begin{defi}
 Let $X$ be a finite space of height $2$ and let $a, b \in X$ be two maximal points. If $U_a\cup U_b$ is contractible, we will say that there is a \textit{qc-reduction} from $X$ to $Y-\{a,b\}$, where $Y=X\cup\{c\}$ with $a,b<c$. The point $c$, which replaces $a$ and $b$ in $Y-\{a,b\}$, is called a \textit{relative} of $a$ and $b$. Figure \ref{qcred} illustrates a qc-reduction.
\end{defi}

\begin{figure}[h]
\noindent\begin{minipage}{0.32\linewidth}
  \begin{displaymath}
\xymatrix@C=10pt{\\
 \bullet a \ar@{-}[d] \ar@{-}[drr] & \bullet b \ar@{-}[d] \ar@{-}[dr] & \bullet \ar@{-}[dl] \ar@{-}[d] \\
 \bullet \ar@{-}[d] \ar@{-}[dr] & \bullet \ar@{-}[dl] \ar@{-}[dr] \ar@{-}[drr] & \bullet \ar@{-}[dll] \ar@{-}[dl] \ar@{-}[d] \ar@{-}[dr]\\
 \bullet & \bullet & \bullet & \bullet}
  \end{displaymath}
\hspace{50pt} $X$
\end{minipage}
\noindent\begin{minipage}{0.32\linewidth}
\vspace{-2.5pt}
  \begin{displaymath}
\xymatrix@C=10pt{
 \bullet c \ar@{-}[d] \ar@{-}[dr]\\
 \bullet a \ar@{-}[d] \ar@{-}[drr] & \bullet b \ar@{-}[d] \ar@{-}[dr] & \bullet \ar@{-}[dl] \ar@{-}[d] \\
 \bullet \ar@{-}[d] \ar@{-}[dr] & \bullet \ar@{-}[dl] \ar@{-}[dr] \ar@{-}[drr] & \bullet \ar@{-}[dll] \ar@{-}[dl] \ar@{-}[d] \ar@{-}[dr]\\
 \bullet & \bullet & \bullet & \bullet}
  \end{displaymath}
\hspace{50pt} $Y$
\end{minipage}
\noindent\begin{minipage}{0.32\linewidth}
\vspace{5.7pt}
  \begin{displaymath}
\xymatrix@C=10pt{\\
 \bullet c \ar@{-}[d] \ar@{-}[dr] \ar@{-}[drr] &   & \bullet \ar@{-}[dl] \ar@{-}[d] \\
 \bullet \ar@{-}[d] \ar@{-}[dr] & \bullet \ar@{-}[dl] \ar@{-}[dr] \ar@{-}[drr] & \bullet \ar@{-}[dll] \ar@{-}[dl] \ar@{-}[d] \ar@{-}[dr]\\
 \bullet & \bullet & \bullet & \bullet} 
  \end{displaymath}
\hspace{40pt} $Y- \{a,b\}$
\end{minipage}
\caption{A qc-reduction.}\label{qcred}
\end{figure}
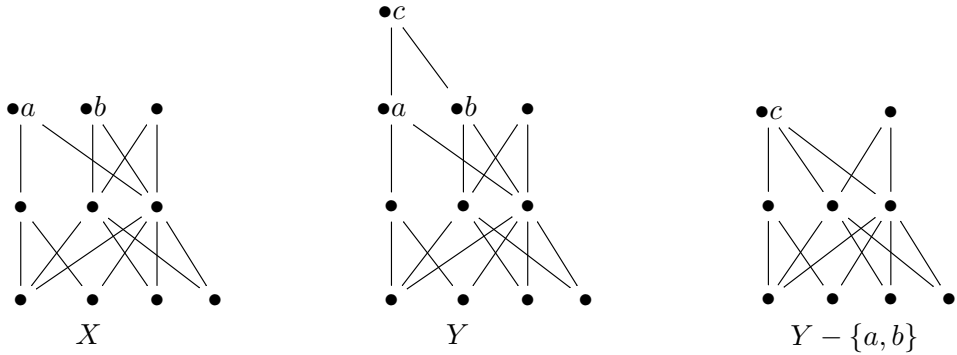

\begin{defi}
 A finite space $X$ of height $2$ is called \textit{qc-reducible} if one can obtain a space with maximum by performing a sequence of qc-reductions.
\end{defi}

\begin{obs}
 In the first step of a qc-reduction, the point $c$ added to $X$ is a weak point of $Y$. The points $a,b$ removed in the second step are beat points of $Y$. It follows that the qc-reductions preserve the weak homotopy type and that qc-reducible spaces are homotopically trivial. Of course, not every homotopically trivial finite space is qc-reducible. However one can show that any contractible finite space of height $2$ is qc-reducible.
\end{obs}

\begin{ej}
 The finite space $X$ of Figure \ref{collnoqc}, which was introduced in \cite{ba}, is collapsible and therefore weak equivalent to a point. However no qc-reduction can be performed on $X$.
\end{ej}

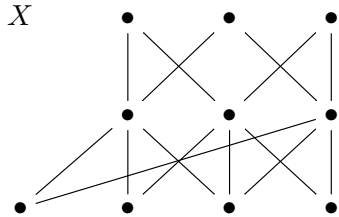
\begin{figure}[h]
 \begin{displaymath}
 \xymatrix@C=10pt{X
& & \bullet \ar@{-}[d] \ar@{-}[drr] & & \bullet \ar@{-}[dll] \ar@{-}[drr] & & \bullet \ar@{-}[dll] \ar@{-}[d]\\
& & \bullet \ar@{-}[dll] \ar@{-}[d] \ar@{-}[drr] & & \bullet \ar@{-}[dll] \ar@{-}[d] \ar@{-}[drr] & & \bullet \ar@{-}[dllllll] \ar@{-}[dll] \ar@{-}[d]\\
  \bullet & & \bullet & & \bullet & & \bullet }
 \end{displaymath}
\caption{A collapsible and non-qc-reducible space.\label{collnoqc}}
\end{figure}

\bigskip

 The proof of the following result, which will be used in the next section, can be found in \cite{ba}. 

\begin{prop}\label{qc-red}
 Let $X$ be a finite space of height at most 2 and such that $H_2(X)=0$. Let $a,b$ be two maximal elements of $X$. Then the following are equivalent.
\begin{itemize}
 \item[1.]$U_a \cup U_b$ is contractible.
 \item[2.]$U_a \cap U_b$ is nonempty and connected.
 \item[3.]$U_a \cap U_b$ is contractible.
\end{itemize}
\end{prop}

Here $H_2(X)$ denotes the second homology group of $X$ with integer coefficients.

By \cite[Thm 11.2.10]{ba}, the associated finite space $\X(K)$ of a $2$-complex $K$ is qc-reducible if and only if $K$ is contractible and \textit{quasi-constructible}. The class of  quasi-constructible complexes contains all constructible complexes. We refer the reader to \cite{bjo, ko} for more details on constructible complexes.

\section{Main Results}

\vspace{5pt}
We will use the homotopy theory of finite spaces to investigate Whitehead's asphericity question. We will focus our attention on the finite case (i.e. counterexamples of type $(a)$) of Howie's theorem \ref{ho} and restate the question in terms of finite spaces.

Suppose that there exists a counterexample of type $(a)$. Since every finite CW-complex is homotopy equivalent to a finite simplicial complex of the same dimension (see \cite[(7.2)]{coh}), it follows that there is a contractible simplicial complex $L$ of dimension 2 and a subcomplex $K=L-\sigma$ with $\sigma$ a 2-simplex of $L$, such that $K$ is connected and non-aspherical. Let $\X(L)$ and $\X(K)$ be their face posets. Then $\X(L)$ is a homotopically trivial space of height 2 and $\X(K)$ is a connected, non-aspherical subspace of $\X(L)$, obtained by removing the maximal point $\sigma\in\X(L)$. This motivates the following conjecture.

\begin{conj}\label{conj_etf}
 Let $X$ be a homotopically trivial finite space of height 2 and let $a\in X$ be a maximal point such that $X-a$ is connected. Then $X-a$ is aspherical.
\end{conj}

It is clear that if conjecture \ref{conj_etf} is valid, then there are no counterexamples of type $(a)$. On the other hand, suppose that the answer to Whitehead's question is \textit{positive}. Given a homotopically trivial finite space $X$ of height 2 and a maximal point $a\in X$  such that $X-a$ is connected, the associated simplicial complex $\K(X)$ is aspherical (in fact, it is contractible) and the subcomplex $\K(X-a)$ of $\K(X)$ is connected. Hence $\K(X-a)$, and therefore also $X-a$, are aspherical. This means that, in this case,  conjecture \ref{conj_etf} is also valid.

\medskip

Note that Whitehead's original question is equivalent, in the case of compact polyhedra, to asking whether every subspace of a finite aspherical space of height two is itself aspherical.

Sometimes it is convenient to work with spaces which are not necessarily path connected.  In general we will say that a space $X$ is aspherical if every path connected component of $X$ is aspherical.

\begin{obs}
 Let $x,y\in X$ be two maximal points. If $X'$ is obtained from $X$ by performing a qc-reduction that does not involve $x$ nor $y$, then $x,y\in X'$ and $U^X_x\cup U^X_y=U^{X'}_x\cup U^{X'}_y$. Here we write $U^X_x,\ U^{X'}_x$ in order to distinguish whether the open subsets are considered in $X$ or $X'$.
\end{obs}

\begin{lema}\label{split}
 Let $X$ be a finite qc-reducible space and let $a\in X$ be a maximal point. Then the qc-reductions can be reordered and split into two phases such that
 \begin{itemize}
 \item[i)] in the first phase, the point $a$ and its relatives are not involved in any reduction,
 \item[ii)] in the second phase every reduction involves a relative of the point $a$.
 \end{itemize}
\end{lema}
\begin{proof}
 Let $X=X_0,X_1,X_2,\dots,X_n$ be a sequence of finite spaces starting with $X$ where each $X_i$ is obtained from $X_{i-1}$ by a qc-reduction $q_i$ and such that $X_n$ has a maximum.
 Let $q_{i_1},q_{i_2},\dots,q_{i_k}$ be the reductions involving $a$ and its relatives. Concretely, $q_{i_1}$ involves $a=a_0$ and $a'_0\in X_{i_1-1}$, which are replaced by $a_1$ in $X_{i_1}$, $q_{i_2}$ involves $a_1$ and $a'_1\in X_{i_2-1}$ which are replaced by $a_2$ in $X_{i_2}$ and so on.
 We will now perform all the reductions which don't involve the relatives of $a$. In order to do that, set $\widetilde{X}_0=X_0,\widetilde{X}_1=X_1,\dots,\widetilde{X}_{i_1-1}=X_{i_1-1}$, then skip the reduction $q_{i_1}$ and set $\widetilde{X}_{i_1}=\widetilde{X}_{i_1-1}$. By the previous remark, the next reduction $q_{i_1+1}$ can be performed on $\widetilde{X}_{i_1}$,
 because it does not involve $a_1$
 (unless $i_1+1=i_{2}$, in which case we  keep skipping reductions). In this way we obtain $\widetilde{X}_{i_1+1}$. Analogously, we can perform the reductions $q_{i_1+2},\dots,q_{i_2-1}$ and skip $q_{i_2}$. We proceed similarly with every $q_{i_j}$ and obtain $\widetilde{X}_n$ when the first phase of reductions is over.

 Again by the previous remark, it is easy to see that the reductions $q_{i_1},q_{i_2},\dots,q_{i_k}$ can now be performed on $\widetilde{X}_n$ to obtain a space with maximum. This is the second phase.
\end{proof}

\begin{lema}\label{pi_1}
 Let $X$ be a finite qc-reducible space of height 2 and let $a\in X$ be a maximal point such that $X-a$ is connected. If every qc-reduction involves a relative of $a$, then $\hat{U}_a$ is connected and the inclusion $\hat{U}_a\hookrightarrow X-a$ induces an isomorphism $\pi_1(\hat{U}_a)\to \pi_1(X-a)$.
\end{lema}
\begin{proof}
 Similarly as in the previous lemma, let $X=X_0,X_1,X_2,\dots,X_n$ be a sequence of finite spaces where each $X_i$ is obtained from $X_{i-1}$ by a qc-reduction $q_i$ which replaces $a_{i-1}$ and $a'_{i-1}$ with $a_i$, where $a=a_0$, and such that $X_n$ has a maximum. The elements $a'_i$ are exactly the maximal points of $X$ other than $a$. Thus $X=U_a\cup\bigcup_{i=0}^{n-1} U_{a'_i}$ and $X-a=\hat{U}_a\cup\bigcup_{i=0}^{n-1} U_{a'_i}$.

 Since $\hat{U}_a\cap U_{a'_0}=U_a\cap U_{a'_0}$ is nonempty and connected, by the van Kampen theorem it follows that the inclusion $\hat{U}_a\hookrightarrow \hat{U}_a\cup U_{a'_0}$ induces an isomorphism in the fundamental groups. Since a second reduction, involving $a_1$ and $a'_1$, can be performed,  we know that $(\hat{U}_a\cup U_{a'_0})\cap U_{a'_1}=(\hat{U}_a\cup \hat{U}_{a'_0})\cap U_{a'_1}=\hat{U}_{a_1}\cap U_{a'_1}=U_{a_1}\cap U_{a'_1}$ is also nonempty and connected. Again by van Kampen, it follows that $\hat{U}_a\hookrightarrow \hat{U}_a\cup U_{a'_0}\cup U_{a'_1}$ induces an isomorphism in the fundamental groups. The result now follows by iterating this reasoning.

Note that, since the intersections $U_{a_i}\cap U_{a'_i}$ are connected, $\hat{U}_a$ is also connected. Otherwise $\hat{U}_a\cup U_{a'_0}$ could not be connected, neither could $\hat{U}_a\cup\bigcup_{i=0}^{k} U_{a'_i}$ for any $1\leq k\leq n-1$.
\end{proof}

We are now ready to prove one of the main results of the article.

\begin{teo}\label{main1}
 Let $X$ be a finite qc-reducible space of height 2 and let $a\in X$ be a maximal point such that $X-a$ is connected. Then $X-a$ is aspherical.
\end{teo}
\begin{proof}
 Applying \ref{split}, let $Y$ be the last space obtained in the first phase of the reductions. Note that $X-a$ is weak equivalent to $Y-a$ by performing the same qc-reductions. Note also that $Y$ is qc-reducible by reductions that involve only the point $a$ and its relatives. By \ref{pi_1}, $\pi_1(Y-a)=\pi_1(\hat{U}_a)$ and this group is free, since $\he(\hat{U}_a)\leq1$.

 Let us now consider the associated complexes of these spaces. $\mathcal{K}(X)$ is contractible and $\mathcal{K}(X-a)$ is a connected subcomplex of $\mathcal{K}(X)$ with free fundamental group. More precisely, $\mathcal{K}(X)=\mathcal{K}(X-a)\cup \st a$ with $\mathcal{K}(X-a)\cap \st a=\lk a$. Here $\st a$ denotes the (closed) star of the vertex $a\in \mathcal{K}(X)$ and $\lk a$ denotes its link. Therefore the CW-complex $\mathcal{K}(X)$ is obtained from $\mathcal{K}(X-a)$ by attaching one 0-cell, and some 1-cells and 2-cells. Let $(\st a)^{(1)}$ be the 1-skeleton of $\st a$ and let $L=\mathcal{K}(X-a)\cup (\st a)^{(1)}$ be the complex between $\mathcal{K}(X-a)$ and $\mathcal{K}(X)$ missing only the 2-cells of $\st a$. Then $L$ is also connected and $\pi_1(L)$ is also free. By \cite[Thm 1]{coc}, $L$ is aspherical. Since $L$ is obtained from  $\mathcal{K}(X-a)$ by attaching only $0$-cells and $1$-cells,   it follows that $\mathcal{K}(X-a)$ is also aspherical.
\end{proof}

As an immediate consequence of this result we deduce the following.

\begin{coro}
 There are no counterexamples of type $(a)$ in the class of quasi-constructible and contractible 2-complexes.
\end{coro}

\subsection{Aspherical reductions}

We investigate now a method of reduction which preserves the asphericity of finite spaces of height $2$. 
Recall that, given a point $x\in X$, we denote by $\hat{C}_x$ the subspace $\hat{C}_x=\{y\in X:\ y<x \text{ or } y>x\}$.

\begin{defi}
Let $X$ be a finite space. A point $x\in X$ is called an \textit{a-point} if $\hat{C}_x$ is a disjoint union of contractible spaces. In that case, the move $X\to X-x$ is called an \textit{a-reduction}. A finite space $X$ of height at most $2$ is called \textit{strong aspherical} if there is a sequence of  a-reductions which transforms $X$ into the singleton.
\end{defi}

Note that the notion of a-point generalizes that of weak point (see Section \ref{preliminaries}). In particular, collapsible finite spaces of height $2$ are strong aspherical. Note also that strong aspherical spaces are not necessarily homotopically trivial. In fact, every finite space of height one is strong aspherical.

\begin{prop} \label{a-points}
 Let $X$ be a finite space of height at most $2$ and let $a\in X$ be an a-point. Then $X$ is aspherical if and only if $X-a$ is aspherical. In particular, strong aspherical spaces are aspherical.
\end{prop}
\begin{proof}
 Since $\K(X)=\K(X-a)\cup a\K(\hat{C}_a)$ and  $\K(\hat{C}_a)$ is a disjoint union of contractible spaces, the space $\K(X)$ is homotopically equivalent to a space obtained from $\K(X-a)$ by attaching a cone of a discrete subspace. Since attaching cells of dimensions 0 and 1 does not change the asphericity, it follows that $\K(X)$ is aspherical if and only if $\K(X-a)$ is.
\end{proof}

The following two examples illustrate how to use this result as a method of reduction to investigate asphericity of finite spaces. Note that, in both cases, the original spaces have no weak points.

\begin{ej}
In the space $X$ of figure \ref{fig:nonasph}, the elements $a$ and $d$ are a-points. The space $Y$ which one obtains removing these points has only two maximal elements, $b$ and $c$. Therefore $Y$ is homotopy equivalent to the \textit{non-Hausdorff suspension} of $U_b\cap U_c$ (see \cite{ba}). Since  $U_b\cap U_c$ is not acyclic, it follows that $Y$ is non-aspherical (in fact its associated complex is homotopy equivalent to $S^2$). This shows that the space $X$ (and consequently, the associated complex $\k(X)$) is non-aspherical.

\begin{figure}[H]
  \begin{displaymath}
  \xymatrix@C=20pt{X
&  \bullet a \hspace{-5pt} \ar@{-}[d] \ar@{-}[drrr] & \bullet b \hspace{-5pt} \ar@{-}[dl] \ar@{-}[d] \ar@{-}[dr] \ar@{-}[drr] & \bullet c \hspace{-5pt} \ar@{-}[dl] \ar@{-}[d] \ar@{-}[dr]\\
&  \bullet \ar@{-}[d] \ar@{-}[dr] & \bullet \ar@{-}[dl] \ar@{-}[d] \ar@{-}[dr] & \bullet \ar@{-}[dl] \ar@{-}[dr]  & \bullet \ar@{-}[dl] \ar@{-}[d] & \bullet d \hspace{-5pt} \ar@{-}[dllll] \ar@{-}[dl]\\
&  \bullet & \bullet & \bullet & \bullet }
  \end{displaymath} 
\caption{}
\label{fig:nonasph}
\end{figure}
\end{ej}

\begin{ej}
In the space $X$ of figure \ref{fig:asph}, the elements $a$, $b$, $c$ and $d$ are a-points. If we remove these points, we obtain a space with minimum, and therefore contractible. It follows that $X$ (and consequently, $\k(X)$) is aspherical.

\begin{figure}[H]
  \begin{displaymath}
  \xymatrix@C=20pt{X
& \bullet \ar@{-}[d] \ar@{-}[dr] & \bullet \ar@{-}[dl]  \ar@{-}[dr] & \bullet \ar@{-}[dl] \ar@{-}[d] \\
\bullet d \hspace{-5pt} \ar@{-}[d] \ar@{-}[dr] \ar@{-}[drrr] & \bullet \ar@{-}[dl] \ar@{-}[d] \ar@{-}[dr] & \bullet \ar@{-}[dl] \ar@{-}[d] \ar@{-}[dr]  & \bullet \ar@{-}[dlll] \ar@{-}[dl] \ar@{-}[d] \\
\bullet a \hspace{-5pt} & \bullet b \hspace{-5pt} & \bullet   & \bullet c \hspace{-5pt} }
  \end{displaymath} 
\caption{}
\label{fig:asph}
\end{figure}
\end{ej}

\begin{teo}\label{strongaspherical}
Let $X$ be a strong aspherical space and let $a\in X$. Then $X-a$ is strong aspherical.
\end{teo}
\begin{proof}
 We proceed by induction on the cardinality of $X$. Choose a sequence of a-reductions from $X$ to the singleton.
 If $a$ is the first a-point to be removed, then $X-a$ is strong aspherical by definition.  Otherwise, let $b\in X$ be the first reduction, with $b\neq a$. Then $X-b$ is strong aspherical and, by induction, $X-\{a,b\}$ is strong aspherical.

On the other hand, since $\hat{C}_b$ is a disjoint union of contractible subspaces, and it is of height at most 1, then $\hat{C}_b-a$  is also a disjoint union of contractible subspaces (possibly empty). This implies that $b$ is also an a-point of $X-a$. Since $(X-a)-b$ is strong aspherical, it follows that $X-a$ is strong aspherical.
\end{proof}

\begin{coro}
 The answer to Whitehead's original question is positive for strong aspherical finite spaces.
\end{coro}

\begin{defi}
 A simplicial complex $K$ is called \textit{strong aspherical} if its face poset $\X(K)$ is strong aspherical.
\end{defi}

By the previous results, the answer to Whitehead's question is positive for strong aspherical simplicial complexes.
Note that this fact can also be proved using the following characterization of strong aspherical complexes.

\begin{prop}
 A simplicial complex $L$ is strong aspherical if and only if $L$ collapses to a one-dimensional subcomplex.
\end{prop}

\begin{proof}
 It is clear that if $L$ collapses to a one-dimensional subcomplex $K$, then $\X(L)$ a-reduces to $\X(K)$, which in turn a-reduces to a point. 

 Conversely, suppose that $\X(L)$ is strong aspherical and that $L$ collapses to a two-dimensional subcomplex $K$ on which no more collapses can be performed. Since $K$ has no collapses, the only possible a-points of $\X(K)$ are minimal points. Given such an a-point $v\in \X(K)$, $\hat{F}_v$ must be discrete, since it is a space of height at most 1 in which every component is contractible and it does not have beat points. This implies that no
element $\sigma \in \X(K)$ of height $2$ can become an a-point and be removed
in a process of a-reductions. This contradicts the fact that,  by \ref{strongaspherical}, the subspace $\X(K)$ of $\X(L)$ is strong aspherical.
\end{proof}

\medskip

{\bf Acknowledgement.}
We would like to thank Jonathan Barmak for useful comments and suggestions.

\end{document}